\documentclass[letterpaper,10pt, journal,twoside]{IEEEtran}  

\usepackage{amsmath,amssymb,color,cite}
\usepackage{graphicx}

\usepackage{nicefrac}
\usepackage{latexsym}
 \usepackage{verbatim}
\usepackage[latin1]{inputenc}
\usepackage{enumerate}
\usepackage{enumitem}
 \usepackage{amsthm}
 \usepackage{xfrac}

\usepackage{pgf}
\usepackage{pgfplots,tikz}
\usetikzlibrary{tikzmark}
\usepackage[font=footnotesize]{caption}
\usepackage{arydshln}

\usepackage{amsfonts}
\usepackage{amsbsy}
\usepackage{bm}

\newcommand{\new}[1]{{\color{blue} #1}}

\DeclareMathOperator{\trace}{tr}

\newcommand\oprocendsymbol{\hbox{$\bullet$}}
\newcommand\oprocend{\relax\ifmmode\else\unskip\hfill\fi\oprocendsymbol}

\newcommand{\abs}[1]{\ensuremath{| #1 |}}
\newcommand{\norm}[1]{\ensuremath{\| #1 \|}}

\let\leq\leqslant
\let\geq\geqslant

\newcommand{\R}{\mathbb R}

\newcommand{\Z}{\mathbb Z}

\newcommand{\calH}{\ensuremath{\mathcal{H}}}

\newcommand{\calL}{\ensuremath{\mathcal{L}}}
\newcommand{\calM}{\ensuremath{\mathcal{M}}}
\newcommand{\calN}{\ensuremath{\mathcal{N}}}

\newcommand{\calP}{\ensuremath{\mathcal{P}}}

\newcommand{\calS}{\ensuremath{\mathcal{S}}}

\newcommand{\calZ}{\ensuremath{\mathcal{Z}}}

\newtheorem{theorem}{Theorem}[section]

\newtheorem{lemma}[theorem]{Lemma}
\newtheorem{corollary}[theorem]{Corollary} \theoremstyle{remark}
\newtheorem{remark}[theorem]{Remark} \theoremstyle{definition}
\newtheorem{assumption}{Assumption}
\newtheorem{definition}[theorem]{Definition}
\newtheorem{example}[theorem]{Example} 

\allowdisplaybreaks
\renewcommand{\baselinestretch}{.97}

\title{When sampling works in data-driven control: Informativity for stabilization
  in continuous time}

\author{Jaap Eising \qquad Jorge Cort\'{e}s\thanks{
    J. Eising is with the Department of Information Technology and Electrical
    Engineering at ETH Z\"{u}rich, Switzerland,
    \texttt{jeising@ethz.ch}. J. Cort\'es is with the Department of Mechanical
    and Aerospace Engineering, UC San
    Diego. \texttt{cortes@ucsd.edu}}}

\begin{document}

\maketitle

\begin{abstract}
  This paper introduces a notion of data informativity for
  stabilization tailored to continuous-time signals and systems. We
  establish results comparable to those known for discrete-time
  systems with sampled data.
  We justify that additional assumptions on the properties of the
  noise signals are needed to understand when sampled versions of
  continuous-time signals are informative for stabilization,
  thereby introducing the notions of square Lipschitzness and total
  bounded variation. This allows us to connect the continuous and
  discrete domains, yielding sufficient conditions to synthesize a
  stabilizing controller for the true continuous-time system on the
  basis of sampled data. Simulations illustrate our results.
\end{abstract}

\vspace{-1.3em}\section{Introduction}


\vspace{-0.3em}Data-driven control has emerged as an appealing way of combining the
use of data with solid theoretical principles from systems theory to
synthesize controllers for unknown systems on the basis of
measurements. The development of `one-shot' controller design methods
in particular has attracted significant interest, where data is
directly employed for design without an intermediate system
identification step. Owing to the discrete-time nature of sampled
data, most of this progress has been for systems operating in
discrete time. However, systems that evolve in continuous time are
widespread across engineering disciplines due to the physical nature
of real-world phenomena. Often times, such systems are interconnected
with digital controllers that operate in discrete time. In the context
of data-driven control, understanding the interface between the
continuous and digital domains is particularly relevant as
measurements come often in the form of samples. The goal of this
paper is to understand to what extent continuous-time data and its 
samples are informative enough to ensure stabilizability of an
unknown plant evolving in continuous time.


\emph{Literature review:} Data-driven control has been particularly
fruitful for linear systems, where the notion of persistency of
excitation and specifically Willems' fundamental
lemma~\cite{JCW-PR-IM-BLMDM:05} have allowed users to express any finite
length trajectory in terms of sufficiently informative measurements.
This has proven useful in a range of problems, including
simulation~\cite{IM-PR:08}, linear feedback design \cite{CDP-PT:20},
predictive control \cite{JC-JL-FD:19}, and optimal control
laws~\cite{GRGBDS-SA-CL-LC:18,GB-VK-FP:19}. Aligned with this body of
work, the informativity approach to data-driven control introduced
in~\cite{HJVW-JE-HLT-MKC:20,HJVW-MKC-JE-HLT:22} considers measurements
that do not contain enough information to obtain a unique system. By
making assumptions on the model class and noise model, this approach
explicitly determines the set of all systems consistent with
the measurements, thereby enabling the certification of desirable
properties (e.g., stabilizability) for the measured system.  Most of
the aforementioned works deal with discrete-time systems, and
correspondingly with measurements consisting of sequences of states
and inputs. To our knowledge, the only works dealing with
continuous-time systems do so on the basis of discretized
measurements, see
e.g.,~\cite{MB-SG-JC:22,JM-MS-22,AB-CDP-PT:20,JB-SW-MH-FA:2021}.
In line with this, \cite{VGL-MAM:22} derives a variant of Willems' lemma
for continuous-time systems on the basis of samples. Moreover, many real-world
phenomena take place in continuous time and as such, the examples of
\cite{JB-CWS-FA:20,HJVW-JE-MKC-HLT:22,HJVW-CDP-MKC-PT:20,CDP-PT:21}
are found by discretizing a continuous-time system.

\emph{Statement of contributions:} We deal with the model class of
continuous-time linear systems and investigate the informativity of
data for stabilization.  First, we provide conditions for
stabilizability with measurements in the form of continuous-time
trajectories and noise models given in terms of integrals of the noise
signal. Complementarily, we also derive conditions of when
  samples of these signals are informative for continuous-time
  stabilization. To keep the focus on the effect of noise, we refrain
  from considering the problem of approximating the derivative from
  samples of the state, and assume that samples of this derivative are
  given. Through an example, we show how no connection between the
two notions can be established without additional assumptions on the
noise model, motivating our consideration of square Lipschitzness and
bounded total square variation noise models. These notions allow us to
establish several connections between the continuous and discrete
domains, culminating in sufficient conditions for the synthesis of a
stabilizing controller for the true continuous-time system on the
basis of sampled data. Finally, we study the role of the sampling
stepsize, provide a bound on it to guarantee the informativity of the
sampled data and a criterion that enables us to remove a portion of
the measurements without losing informativity. Simulations illustrate
our results.

\vspace{-1em}
\section{Problem formulation}\label{sec:problem}
\vspace{-0.4em}Consider\footnote{We denote by $\Z_{>0}$ and $\R$ the set of positive
  integer and real numbers, resp.
  For a vector $v\in\R^n$ and a matrix $A\in\R^{n\times n}$,
  $\norm{v}$ and $\norm{A}$ denote the Euclidean norm and induced
  Euclidean norm, resp. The Moore-Penrose pseudo-inverse of $A$ is
  denoted $A^\dagger$. We let $I_n$ denote the $n \times n$ identity
  matrix.
  A property holds for almost all $t\in[0,T]$ if the set for which the
  property does not hold has Lebesgue measure $0$.  A function
  $z:[0,T]\rightarrow \R^n$ is $L$-Lipschitz if
  $\norm{z(t_1)-z(t_2)} \leq L \abs{t_1-t_2}$ for $t_1,t_2 \in
  [0,T]$. If $z$ is differentiable, this is equivalent to
  $\norm{z'(t)} \leq L$ for all $t \in [0,T]$.  $z$ is absolutely
  continuous if there is an integrable function
  $\hat{z}:[0,T]\rightarrow \R^n$ such that $
  z(t) = z(0) + \int_0^t \hat{z}(\tau) d\tau.$
Note that this means that $z$ has a derivative $\hat{z}$ almost
everywhere. We denote the set of square-integrable functions
by~$\calL_2$.}  the continuous-time system
\vspace{-0.3em}\begin{equation}\label{eq:system}
  \dot{{x}}(t) = A_s {x}(t)+ B_s
  {u}(t)+ {w}(t),
\end{equation}
where $x(t)\in\R^n$ is the state, ${u}(t)\in\R^m$ is the input, and
${w}(t) \in \R^n$ is a disturbance. Here $A_s:\R^n\rightarrow \R^n$
and $B_s:\R^m\rightarrow\R^n$ are unknown linear maps, and the
sub-index $s$ is used to denote the true \emph{system} matrices.
Given a finite-time horizon $T>0$, we are interested in
\textit{absolutely continuous} state trajectories ${x}$ of~\eqref{eq:system} on the interval~$[0,T]$.

Since $A_s$ and $B_s$ are unknown, we take an approach similar to
data-driven approaches to determine properties of the system and
control it.  We consider continuous-time measurements over the
interval $[0,T]$. Specifically, we consider measured state
$x:[0,T]\rightarrow\R^n$ and input $u:[0,T] \rightarrow \R^m$
trajectories.  We assume that the associated disturbance
$w:[0,T]\rightarrow\R^n$ satisfies a \textit{noise model}, denoted
$\Delta$, defined as follows: for $0 \le Q\in\R^{n\times n}$,
$w\in\Delta$ if and only if
\vspace{-0.3em}\begin{equation}\label{eq:noise model} \int_{0}^T
  w(t)w(t)^\top dt \leq Q.
\end{equation}
Taking the trace of both sides, we see that~\eqref{eq:noise model}
implies $ \int_{0}^T w(t)^\top w(t) dt \leq \trace(Q)$, and therefore
$\Delta\subseteq\calL_2$.

This noise model captures the behavior of common assumptions on noise
signals. For instance, if for almost all $t\in[0,T]$,
\begin{equation}\label{eq:pointwise noise}
  w(t)w(t)^\top \leq
  \tfrac{1}{T} Q,
\end{equation}
then \eqref{eq:noise model} holds.
If we assume a bound on the norm of the values of the
disturbance signal, this can be brought into this form by noting
$w(t)^\top w(t) \leq k$ if and only if $w(t)w(t)^\top \leq kI_n$.
We make the following assumption on the measurements.

\vspace{-0.4em}\begin{assumption}[Well-behavedness of the
  measurements]\label{ass:state-input}
  The measurement signals $x:[0,T]\rightarrow\R^n$,
  $u:[0,T]\rightarrow\R^m$ satisfy
  \begin{itemize}
  \item The state signal $x$ is absolutely continuous;
  \item The input signal $u$ is square integrable;
  \item The corresponding noise signal
    $w:[0,T]\rightarrow\R^n$ belongs to $\Delta$ as defined
    by~\eqref{eq:noise model};
  \item The triplet $(x,u,w)$ satisfies~\eqref{eq:system} for almost
    all $t\in [0,T]$.
  \end{itemize}
\end{assumption}

This assumption is mild but necessary for our ensuing analysis.  Since
$x$ is absolutely continuous on the compact interval $[0,T]$, it is
bounded. As a consequence, $x \in \calL_2$. This, together with the
fact that~\eqref{eq:system} holds almost everywhere and
$\Delta\subseteq\calL_2$, implies that $\dot{x}\in \calL_2$ too.

Underlying the informativity approach is the observation that, on the
basis of measurements, one can only conclude a property of interest of
the true system $(A_s,B_s)$ if \textit{all} systems compatible with
the measurements have such property. As such, we consider the set of
all systems compatible with the measurement and noise model as defined
by
\vspace{-0.3em}\[
  \Sigma = \{ (A,B) \in \R^{n\times n} \times
  \R^{n\times m} \mid \dot{x}-Ax-Bu \in \Delta\}.
\]
We are interested in finding a stabilizing controller for $(A_s,B_s)$
on the basis of the measurements $x$ and $u$. This leads to the
following notion.

\vspace{-0.4em}\begin{definition}[Informativity of continuous-time data for quadratic
  stabilization]\label{def:inf ctqs}
  Data consisting of state $x:[0,T]\rightarrow\R^n$ and input
  $u:[0,T]\rightarrow\R^m$ trajectories are
  \textit{informative for quadratic stabilization} if and only if
  there exists $K\in\R^{m\times n}$ and
  $P\in \R^{n\times n}$ such that $P>0$ and for all $(A,B)\in \Sigma$:
\vspace{-0.2em} \begin{equation}\label{eq:Lyap}
    (A+BK)P+P(A+BK)^\top <0.\vspace{-0.2em}
  \end{equation}
\end{definition}

Our first objective is to provide necessary and sufficient conditions
on the data $(x,u)$ which ensure this notion of informativity is
satisfied. Our second objective seeks to understand when sampled
versions of the continuous-time data remain informative enough for
stabilization. To formalize this objective, assume we have access to
samples of the signals $x$ and $u$ at a number of discrete time-instants. 
We assume that the stepsize $\delta$ is a whole fraction of
the time horizon\footnote{The choice of a uniform stepsize makes the
  notation simpler, but our results can be easily adapted to deal with
  more general sampling schemes.}, that is,
$\tfrac{T}{\delta}\in\Z_{>0}$, which means that we consider samples at
time instances
$\{t_k = k \delta \}_{k=0}^{{T}/{\delta}-1} \subset [0,T]$.  We
collect the measurements and samples of the noise signal into matrices
\begin{subequations}\label{eq:matrices}
  \begin{align}
    \hspace{-1em}\dot{X}_\delta &= \begin{bmatrix} \dot{x}(0) & \!\!\!\cdots\!\!\! &
      \dot{x}(T-\delta)\end{bmatrix}\!\!,
    &\!\!\!
    X_\delta &= \begin{bmatrix} x(0) & \!\!\!\cdots\!\!\! &
      x(T-\delta)\end{bmatrix}\!\!,
    \\ 
    \hspace{-1em}U_\delta &= \begin{bmatrix} u(0) & \!\!\!\cdots\!\!\! &
      u(T-\delta)\end{bmatrix}\!\!,
    &\!\!\!
    W_\delta &= \begin{bmatrix} w(0)\! & \!\!\!\cdots\!\!\! & w(T-\delta)\end{bmatrix}\!\!.
  \end{align}
\end{subequations}

  \begin{remark}[Samples of the derivative] We focus here on (the
    effect of noise on) the difference in informativity of continuous
    signals and their discretizations. To investigate this, we assume
    that we have access to samples of the derivative~$\dot{x}$, which
    is a common assumption in the literature, see
    e.g.~\cite{MB-SG-JC:22,JM-MS-22,AB-CDP-PT:20,JB-SW-MH-FA:2021}.
    In reality, discrete measurements of this signal are seldom 
    available, which requires estimation from the
    samples of~$x$. However, taking into account such estimation errors,
    while important, is outside of the scope of this paper. \oprocend
  \end{remark}

We assume that $\dot{X}_\delta$, $X_\delta$, and $U_\delta$ are known, but
the samples of the noise, collected in the matrix $W_\delta$, are
unknown. However, we assume $W_\delta$ satisfies some noise model
$\Delta_{\textrm{disc}}$. In particular, as a special case of noise
models considered in the discrete-time informativity
literature~\cite{HJVW-MKC-JE-HLT:22}, we assume that for some
$0 \le Q\in\R^{n\times n}$, $W_\delta \in \Delta_{\textrm{disc}}$ if
and only if
\begin{equation}\label{eq:noise disc}
  \delta W_\delta W_\delta^\top
  \new{=\sum_{k=0}^{{T}/{\delta}-1}  \delta w(k\delta)w(k\delta))^\top }
  \leq Q.
\end{equation}
Note this holds for example if \eqref{eq:pointwise noise} is satisfied for all
$t=k\delta$, where $k=0,\ldots , \sfrac{T}{\delta}-1$.  On the basis
of the samples, we seek to find a stabilizing controller for all
systems in the set
\[
  \Sigma^\delta = \{ (A,B)\in\R^{n\times n}\times \R^{n\times m} \mid
  \dot{X}_\delta -AX_\delta-BU_\delta\in\Delta_{\textrm{disc}} \}.
\]
Our second objective can then be formalized as: provide conditions on
the continuous-time measurements under which we can compare
stabilizability properties of $\Sigma$ and~$\Sigma^\delta$. We focus
on understanding when the continuous-time measurements $(x,u)$ are
informative for quadratic stabilization on the basis of sampled data
and on the stepsizes that make this happen.

\vspace{-0.8em}\section{Data informativity in continuous time}

Here we provide characterizations for when data, either in the form of
continuous-time trajectories or sampled versions of it, is informative
for continuous-time stabilization.
\vspace{-0.8em}\subsection{Informativity with continuous-time data}\label{sec:ct}
Here we address the first objective laid out in
Section~\ref{sec:problem} and characterize when continuous-time data
is informative for stabilization.  We start by observing that the set
$\Sigma$ of systems compatible with the data can be defined via a
Quadratic Matrix Inequality (QMI). Formally, consider measurements $x$
and $u$ satisfying Assumption~\ref{ass:state-input}, with noise
model~\eqref{eq:noise model}.  For
$N=N^\top\in \R^{(2n+m)\times (2n+m)}$, let
\vspace{-0.3em}\[
\calZ(N) := \left\lbrace (A,B) \mid
\begin{bmatrix}I_n &A& B\end{bmatrix}
N
\begin{bmatrix}I_n &A& B\end{bmatrix}^\top
\geq 0 \right\rbrace.
\]
Then, one has $\Sigma= \calZ(N_{\textup{cont}}(Q))$, where
\vspace{-0.3em}\begin{equation}\label{eq:N noise}
  N_{\textup{cont}}(Q)
  :=\begin{bmatrix} Q &0 & 0\\0&0&0\\0&0&0 \end{bmatrix}
  -\int_0^T \begin{pmatrix} \dot{x}(t) \\
    -x(t)\\-u(t)\end{pmatrix}\begin{pmatrix} \dot{x}(t) \\
    -x(t)\\-u(t)\end{pmatrix}^{\!\!\top} \hspace*{-1ex} dt.
\end{equation}
On the other hand, the stability condition~\eqref{eq:Lyap} is equivalent to
\begin{equation}\label{eq:stab informativity}
  \begin{bmatrix} I_n
    \\A^\top \\ B^\top \end{bmatrix}^\top\begin{bmatrix} 0 & -P&
    -PK^\top \\ -P & 0 &0 \\ -KP &0 &0 \end{bmatrix}\begin{bmatrix} I_n
    \\A^\top \\ B^\top \end{bmatrix}>0 ,
\end{equation}
for all $(A,B) \in \Sigma$. Now, both \eqref{eq:stab
    informativity} and $\Sigma$ are given as solution sets of
  QMI's. We can therefore rephrase that the data is informative for
  quadratic stabilization (Definition~\ref{def:inf ctqs}) if all
  $(A,B)$ which satisfy the QMI determined by~\eqref{eq:N noise} also
  satisfy the QMI given by~\eqref{eq:stab informativity}. Such an
  inclusion can be resolved efficiently, as stated in the next
  result.

\begin{theorem}[Necessary and sufficient conditions for informativity
  of continuous-time data]\label{thm:cont stab}
  Suppose that the state and input trajectories
  $x:[0,T]\rightarrow\R^n$ and
  $u:[0,T]\rightarrow\R^m$ satisfy
  Assumption~\ref{ass:state-input}. Then the data $(x,u)$ are
  informative for quadratic stabilization if and only if there exists
  $K\in\R^{m\times n}$, $P\in \R^{n\times n}$, and
  $\beta>0$ such that $P>0$ and
  \begin{equation}\label{eq:LMI stab}
    -\!\!\begin{bmatrix} Q+\beta
      I_n \!\!\!& \!\!P\!\!& \!PK^\top \\ P & 0 &0 \\ KP &0
      &0 \end{bmatrix}\!\! +\int_0^T\!\!\!\begin{pmatrix} \dot{x}(t) \\
      -x(t)\\-u(t)\end{pmatrix}\!\!\!\begin{pmatrix} \dot{x}(t) \\
      -x(t)\\-u(t)\end{pmatrix}^{\!\!\top} \!dt \!\geq \!0.
  \end{equation}
\end{theorem}
\begin{proof}
  We partition $N_{\textup{cont}}(Q)$ as
\vspace{-0.3em}  \[
    N_{\textup{cont}}(Q) =:\begin{bmatrix} N_{11} & N_{12} \\ N_{21} &
      N_{22} \end{bmatrix},
  \] 
  where $N_{11} \in \R^{n\times n}$ and
  $N_{22} \in \R^{(n+m)\times(n+m)}$.  Let
\vspace{-0.3em}  \[
    M:=\begin{bmatrix} 0 & -P& -PK^\top \\ -P & 0 &0 \\ -KP &0
      &0  \end{bmatrix},
  \] 
  for which we consider a similar partition. To prove the result
    we employ~\cite[Corollary 4.13]{HJVW-MKC-JE-HLT:22}, which
    provides conditions equivalent to the required set-inclusion. We
    start by verifying the hypotheses of~\cite[Corollary
    4.13]{HJVW-MKC-JE-HLT:22}, which in this case take the form:
  $N_{22}\leq 0$, $N_{11}-N_{12}N_{22}^\dagger N_{21}\geq 0$,
  $M_{22}\leq0$, and $\ker N_{22} \subseteq \ker N_{12}$. The first
  condition follows from
  \[
    \int_0^T \begin{pmatrix} x(t)\\u(t)\end{pmatrix}\begin{pmatrix}
      x(t)\\u(t)\end{pmatrix}^\top dt\geq 0,
  \] 
  as it is the integral of positive semidefinite
  matrices. The second hypothesis can be derived from
    the fact that $\calZ(N_{\textup{cont}}(Q))$ is nonempty (see
  \cite[Eq. (3.5)]{HJVW-MKC-JE-HLT:22}).  The third is immediate since
  $M_{22}=0$. To show the fourth condition, we need to prove
  \[
    \ker \int_0^T
    \begin{pmatrix}
      x(t)
      \\
      u(t)
    \end{pmatrix}
    \begin{pmatrix}
      x(t)
      \\
      u(t)
    \end{pmatrix}^\top
    dt
    \subseteq
    \ker \int_0^T
    \dot{x}
    \begin{pmatrix}
      x(t)
      \\
      u(t)
    \end{pmatrix}^\top dt .
  \]
  Let
  $v \in \ker \int_0^T \begin{pmatrix}
    x(t)\\u(t)\end{pmatrix}\begin{pmatrix}
    x(t)\\u(t)\end{pmatrix}^\top dt$. Then
  \begin{align*}
    0
    & = v^\top\left( \int_0^T \begin{pmatrix}
        x(t)\\u(t)\end{pmatrix}\begin{pmatrix}
        x(t)\\u(t)\end{pmatrix}^\top dt\right)v
    \\
    & =
      \int_0^T v^\top \begin{pmatrix}
        x(t)\\u(t)\end{pmatrix}\begin{pmatrix}
        x(t)\\u(t)\end{pmatrix}^\top v dt .    
  \end{align*}
  Therefore,
  $ \left(\!\begin{smallmatrix}
      x(t)\\u(t)\end{smallmatrix}\!\right)^{\!\top} v = 0$ for almost
  all $t$, and hence
  $v \in \ker \int_0^T \dot{x}\left(\!\begin{smallmatrix}
      x(t)\\u(t)\end{smallmatrix}\!\right)^{\!\top} dt$. Since
    the hypotheses hold, we can now
    invoke~\cite[Corollary~4.13]{HJVW-MKC-JE-HLT:22} to conclude
    that~\eqref{eq:stab informativity} holds for all $(A,B)\in\Sigma$
  iff there exists $\alpha\geq0$ and $\beta>0$ such that
  \[
    M-\alpha N_{\textup{cont}}(Q) \geq \begin{bmatrix} \beta I_n &0\\
      0& 0\end{bmatrix}.
  \]
  Since $M\not \geq 0$, this requires $\alpha\neq0$. Therefore, we can
  scale $\beta$ and $P$ by $\alpha$, proving the statement.
\end{proof}

As presented, inequality~\eqref{eq:LMI stab} is not a linear matrix
inequality (LMI) in the variables $K$, $P$, and $\beta$. However it
can be rewritten as an LMI using the substitution $L:= KP$. This
allows us to efficiently check for informativity by checking
feasibility of an LMI in the variables $L$, $P$, and
$\beta$. Afterwards, one can use the equation $K=LP^{-1}$ to find the
corresponding stabilizing feedback.

 \begin{remark}[Persistency of excitation]
    The continuous-time signal $x(t)$ is \textit{persistently
      exciting} if there exists $\alpha,T>0$ such that
    $\int_\tau^{\tau+T} x(t)x(t)^\top dt > \alpha I$, for all
    $\tau\geq 0$. Note that, for inequality \eqref{eq:LMI stab} to
    hold with $P>0$, we require $\int_0^T x(t)x(t)^\top dt > 0$. Thus,
    one can say that a necessary condition for informativity for
    quadratic stabilization is that the specific time window of $x(t)$
    is \textit{sufficiently} exciting. \oprocend
\end{remark}

\begin{remark}[Comparison of computational complexity with
  discrete-time case]
  The condition~\eqref{eq:LMI stab} of Theorem~\ref{thm:cont stab}
  takes the form of the scalar inequality $\beta>0$, the $n\times n$
  LMI $P>0$ and an LMI of dimensions $(2n+m) \times (2n+m)$.  Instead,
  the condition of informativity for quadratic stabilization in the
  discrete-time case, cf.~\cite[Theorem 5.1]{HJVW-MKC-JE-HLT:22},
  requires $\beta>0$, $P>0$, and an LMI of dimensions
  $(3n+m)\times (3n+m)$. \oprocend
\end{remark}

\begin{remark}[General noise
	models]\label{rem:assumption-noise-general}
	One can extend Theorem~\ref{thm:cont stab} for noise models more
	general than~\eqref{eq:noise model} without significant additional
	effort.  Let $\Pi: [0,T] \rightarrow \R^{(n+1)\times(n+1)}$ be a
	matrix-valued function and partition it as
	\[
	\Pi(t) = \begin{bmatrix} \Pi_{11}(t) & \Pi_{12}(t) \\ \Pi_{21}(t)
		& \Pi_{22}(t)\end{bmatrix}, \textrm{ with } \Pi_{11}(t) \in
	\R^{n\times n}.
	\]
	Consider the generalized noise model: $w\in\Delta$ if and only if
	\[
	\int_{0}^T
	\begin{bmatrix}
		I_n
		\\
		w(t)^\top
	\end{bmatrix}^\top
	\Pi(t)
	\begin{bmatrix}
		I_n
		\\
		w(t)^\top
	\end{bmatrix}
	dt \geq 0.
	\]
	Under more general assumptions than those made above, 
	an extension of Theorem~\ref{thm:cont stab} can be derived
        analogously. \oprocend 
\end{remark}

%
	
\vspace{-1.5em}\subsection{Informativity with sampled data}\label{sec:sampling}

Here, we analyze when sampled versions of continuous-time data are
sufficiently informative for stabilization.  Let the state and input
trajectories $x:[0,T]\rightarrow\R^n$ and $u:[0,T]\rightarrow\R^m$
satisfy Assumption~\ref{ass:state-input}. Recall the definitions of
the matrices $\dot{X}_\delta$, $X_\delta$, $U_\delta$, and $W_\delta$
in \eqref{eq:matrices}, and consider a noise model
$\Delta_{\textrm{disc}}$ as in \eqref{eq:noise disc}. Note that the 
set of systems compatible with the sampled data
$(\dot{X}_\delta, X_\delta,U_\delta)$ and noise model \eqref{eq:noise
  disc} can be described by $\Sigma^\delta=\calZ(N_\delta(Q))$, where 
\vspace{-0.3em}\[
  N_\delta(Q) := \begin{bmatrix} Q &0 & 0\\0&0&0\\0&0&0 \end{bmatrix}
  - \delta\begin{pmatrix} \dot{X}_\delta \\
    -X_\delta\\-U_\delta\end{pmatrix}\begin{pmatrix} \dot{X}_\delta \\
    -X_\delta\\-U_\delta\end{pmatrix}^\top.
\]
As before, we are interested in finding a stabilizing controller for
$(A_s,B_s)$ on the basis of the discrete measurements, leading to the
following notion.

\vspace{-0.3em}\begin{definition}[Informativity of discrete-time data for quadratic
  stabilization, cf.~{\cite[Def. 2.1]{HJVW-MKC-JE-HLT:22}}]
  The sampled data $(\dot{X}_\delta, X_\delta,U_\delta)$ are
  \textit{informative for continuous-time quadratic stabilization} if
  and only if there exists $K\in\R^{m\times n}$ and
  $P\in \R^{n\times n}$ such that for all $(A,B)\in \Sigma^\delta$:
 \vspace{-0.2em} \[
    P>0, \quad (A+BK)P+P(A+BK)^\top <0.\vspace{-0.2em}
  \]
\end{definition}

We now provide a characterization for informativity of
discrete-time data for stabilization of continuous-time systems:

\begin{theorem}[Necessary and sufficient conditions for
  informativity of discrete-time data]\label{thm:disc stab}
  Suppose the data $(\dot{X}_\delta, X_\delta,U_\delta)$ sampled from
  the system \eqref{eq:system} correspond to noise model
  \eqref{eq:noise disc}. Then, the data
  $(\dot{X}_\delta, X_\delta,U_\delta)$ are informative for
  continuous-time quadratic stabilization if and only if there exists
  $K\in\R^{m\times n}$, $P\in \R^{n\times n}$, and $\beta>0$ such that
  $P>0$ and
  \begin{equation}\label{eq:LMI disc}
    \begin{bmatrix} -Q-\beta I_n \!&
      \!-P\!& \!-PK^\top \\ -P & 0 &0 \\ -KP &0 &0 \end{bmatrix}
    +\delta\!\begin{bmatrix} \dot{X}_\delta \\
      -X_\delta\\-U_\delta\end{bmatrix}\!\!\begin{bmatrix}
      \dot{X}_\delta \\ -X_\delta\\-U_\delta\end{bmatrix}^\top\! \geq
    0. \end{equation}
\end{theorem}

The proof of this result is similar to that of Theorem~\ref{thm:cont
  stab} and we omit it for brevity.  Note that
Theorem~\ref{thm:disc stab} complements the result in
Theorem~\ref{thm:cont stab}, which characterizes informativity of
continuous-time data for stabilization of continuous-time
systems. Together with~\cite[Thm. 5.1]{HJVW-MKC-JE-HLT:22}, which
characterizes informativity of discrete measurements for stabilization
of discrete systems, these paint a complete picture.

Given these characterizations, a natural question is to figure out the
relationship between continuous-time data $(x,u)$ being informative,
as in Theorem~\ref{thm:cont stab}, and sampled versions of it being
informative, as in Theorem~\ref{thm:disc stab}. As it turns out,
without additional assumptions, there is no implication between the
two notions: data $(x,u)$ can meet the condition~\eqref{eq:LMI stab}
but not those in~\eqref{eq:LMI disc}, and vice versa.  The reason for
this can be tracked back to comparing the terms
\begin{equation}\label{eq:comparison}
  \int_0^T\!\begin{pmatrix} \dot{x}(t) \\
    -x(t)\\-u(t)\end{pmatrix}\!\!\begin{pmatrix} \dot{x}(t) \\
    -x(t)\\-u(t)\end{pmatrix}^\top\! dt \textrm{ and }
  \delta\!\begin{bmatrix} \dot{X}_\delta \\ 
    -X_\delta\\-U_\delta\end{bmatrix}\!\!\begin{bmatrix}
    \dot{X}_\delta \\ -X_\delta\\-U_\delta\end{bmatrix}^\top.
\end{equation}
The issue at hand stems from the fact that, if the signal $w$ (or
equivalently the measurement signals $\dot{x}$, $x$, or $u$) is
changed on a measure zero set, the integral on the left remains the
same, whereas the individual samples on the right might change.  

\begin{example}[Comparing noise models]\label{ex:comparing-noise} 
  The comparison of the quantities in~\eqref{eq:comparison} is
  challenging, as we illustrate here.  For the system with noise,
  $ \dot{{x}}(t) = {w}(t)$, consider measurements over the time
  interval $[0,2]$. Let $ \calS_0 = (1,2)$, $\calS_1 = (1,2]$,
  $\calS_2 = [1,2]$, with corresponding noise signals, \vspace{-0.3em}
  \[ w_\alpha(t) = \small\begin{cases} 1 & \textrm{ for } t \in
      \calS_\alpha \\ 0 & \textrm{ otherwise}\vspace{-0.3em}
    \end{cases} \vspace{-0.3em}
  \]
  for $\alpha\in\{0,1,2\}$. Note that, for each $\alpha$,
  $ \int_{0}^3 w_\alpha(t)w_\alpha(t)^\top dt = 1$.  Given initial
  condition $x(0)=1$, each of these noise signals leads to the same
  state trajectory~$x(t)$. Suppose we sample the system at $t=0$,
  $t=1$ and $t=2$. Defining matrices $W_\alpha$ as
  in~\eqref{eq:matrices} corresponding to the noise signals
  $w_\alpha$, resp., we obtain \vspace{-0.3em}
  \[ W_0W_0^\top =0, \quad W_1W_1^\top = 1, \quad W_2W_2^\top
    =2.\vspace{-0.3em}
  \]
  More generally, this shows that without making further assumptions
  on the signal $w$, we cannot necessarily conclude that certain
  bounds hold for the sampled data. \oprocend
\end{example}

\vspace{-1em}\section{Linking informativity of continuous and discrete
  measurements}\label{sec:links}

In this section we study the relationship between
informativity for stabilization of continuous and discrete
measurements.

\vspace{-0.8em}\subsection{Connections between noise
  models}\label{sec:connection-noise-models} As illustrated by
Example~\ref{ex:comparing-noise}, we need to make additional
assumptions on the noise signal to link informativity of continuous
and discrete measurements. Here, we consider two alternative models:
square Lipschitzness and bounded total square variation.

 \vspace{-0.4em}\begin{definition}[Square Lipschitzness]
  For $L\geq 0$, $w:[0,T]\rightarrow\R^n$ is $L$-\textit{square
    Lipschitz} if for all $t_1,t_2\in [0,T]$:
  \begin{equation}\label{eq:noise Lipschitz}
    \norm{w(t_1)w(t_1)^\top -
      w(t_2)w(t_2)^\top} \leq L\abs{t_1-t_2}.
  \end{equation}	
\end{definition}

This property can be guaranteed on the basis of common assumptions on
the signal $w$.

 \vspace{-0.3em}\begin{lemma}[Square Lipschitzness from common
  assumptions]\label{lem:bd+lip implies sq lip}
  Let $w:[0,T]\rightarrow\R^n$ be differentiable, bounded and
  Lipschitz, that is, $\norm{w(t)}\leq L_1$ and
  $\norm{\dot{w}(t)} \leq L_2$ for all $t\in[0,T]$. Then $w$ is
  $2L_1L_2$-square Lipschitz.
\end{lemma}
\begin{proof}
  For $t_1,t_2\in [0,T]$, using that $w$ is differentiable,
  \vspace{-0.3em} \begin{align*}
    w(t_1)w(t_1)^\top\! -\! w(t_2)w(t_2)^\top
    &\!=
      \!\int_{t_2}^{t_1} \frac{d}{dt}\left(w(t)w(t)^\top\right) dt,
    \\
    &\! =\! \int_{t_2}^{t_1} \dot{w}(t)w(t)^\top \!\!+
      w(t)\dot{w}(t)^\top dt. \vspace{-0.3em}
  \end{align*}
  Thus
  $ \norm{w(t_1)w(t_1)^\top -w(t_2)w(t_2)^\top} \leq \int_{t_2}^{t_1}
  2 \norm{\dot{w}(t)w(t)^\top}dt$.  The result follows by noting that
  $ \norm{ \dot{w}(t)w(t)^\top} \leq L_1L_2$. \vspace{-0.4em}
\end{proof}

Note that the conditions of Lemma~\ref{lem:bd+lip implies sq lip} are
not necessary. In particular, $w$ need not be differentiable
everywhere. The following result establishes a relationship between
continuous- and discrete-time noise models.

 \vspace{-0.3em}\begin{lemma}[Continuous- and discrete-time noise models under square
  Lipschitzness]\label{lem:noise models}
  Suppose that $w:[0,T]\rightarrow \R^n$ is $L$-square Lipschitz and
  $\delta$ is such that $\tfrac{T}{\delta}\in\Z_{>0}$. Then
 \vspace{-0.3em}  \[
    \norm{\int_0^T w(t)w(t)^\top dt - \delta W_\delta W_\delta^\top }
    \leq \tfrac{1}{2}\delta TL. \vspace{-0.3em}
  \]
\end{lemma}
\begin{proof}
  Note that we can write \phantom{\qedhere}
 \vspace{-0.3em}  \begin{align}\label{eq:aux}
    &\int_0^T w(t)w(t)^\top dt - \delta W_\delta W_\delta^\top \notag
    \\
    &= \sum_{k=0}^{\sfrac{T}{\delta}-1} \int_{k\delta}^{(k+1)\delta}
      (w(t)w(t)^\top-w(k\delta)w(k\delta)^\top) dt.   \vspace{-0.4em}
  \end{align} 
  Since $w$ is $L$-square Lipschitz,
  $\norm{ w(t)w(t)^\top- w(k\delta)w(k\delta)^\top} \leq
  \abs{t-k\delta} L$ for $t\in [k\delta, (k+1)\delta]$.~Hence,
 \vspace{-0.3em}  \begin{align*}
    &\norm{\int_0^T w(t)w(t)^\top dt - \delta W_\delta W_\delta^\top }
    \\ 
    & \leq \sum_{k=0}^{\sfrac{T}{\delta}-1}
      \int_{k\delta}^{(k+1)\delta}
      \norm{w(t)w(t)^\top-w(k\delta)w(k\delta)^\top} dt
    \\
    & \leq  L
      \sum_{k=0}^{\sfrac{T}{\delta}-1}\int_{k\delta}^{(k+1)\delta}
      \abs{t-k\delta}dt  = \tfrac{1}{2}\delta TL  .\vspace{-0.4em}\tag*{\qed}
  \end{align*} 
\end{proof}
Square Lipschitzness requires the noise signal $w$ to be
continuous. As an alternative, the following concept allows us to consider
discontinuous signals.

 \vspace{-0.3em}\begin{definition}[Total square variation]
  Let $\calP$ denote the set of all partitions of $[\tau,T]$, that is,
  \[
    \calP = \{ \pi = \{t_0,\ldots, t_{n_\pi}\} \mid \tau=t_0\leq \ldots
    \leq t_{n_\pi}=T\}.
  \]
  The \textit{total square variation} of the signal
  $w:[\tau,T]\rightarrow\R^n$ is
  \[
    V_\tau^T(w) = \sup_{\pi\in\calP} \sum_{i=0}^{n_\pi-1}
    \norm{w(t_{i+1})w(t_{i+1})^\top - w(t_{i})w(t_{i})^\top} .
  \]
\end{definition}

The step function is an example of a discontinuous signal that has a
finite total square variation. The following result establishes another 
relationship between continuous- and discrete-time noise models.

\begin{lemma}[Continuous- and discrete-time noise models under bounded
  total variation]\label{lem:BV noise models}
  Suppose that $w:[0,T]\rightarrow \R^n$ has $V_0^T(w)$ finite and let
  $\delta$ be such that $\tfrac{T}{\delta}\in\Z_{>0}$. Then,
\vspace{-0.3em}  \[
    \norm{\int_0^T w(t)w(t)^\top dt - \delta W_\delta W_\delta^\top }
    \leq \delta V_0^T(w).\vspace{-0.3em}
  \]
\end{lemma}
\begin{proof}
  Let $V_k := V_{k\delta}^{(k+1)\delta}(w)$. By definition,
  $ V_0^T(w) = \sum_{k=0}^{\sfrac{T}{\delta}-1} V_k$.  Now, for any
  $t \in [k\delta,(k+1)\delta]$, consider the partition
  $\{k\delta, t,(k+1)\delta\}$. Then,
  \begin{align*}
    V_k &\geq
          \norm{w(t)w(t)^\top-w((k+1)\delta)w((k+1)\delta)^\top}
    \\
        &\quad\quad  + \norm{w(t)w(t)^\top-w(k\delta)w(k\delta)^\top}
    \\
        & \geq
          \norm{w(t)w(t)^\top-w(k\delta)w(k\delta)^\top}.
  \end{align*}
  This, combined with~\eqref{eq:aux}, yields the result.\vspace{-0.3em}
\end{proof}

Note that if $w$ is $L$-square Lipschitz, then $V_0^T(w) \leq LT$, and
in this case the result in Lemma~\ref{lem:BV noise models} (bound with
$\delta LT$) is weaker than that of Lemma~\ref{lem:noise
  models} (bound with $\tfrac{1}{2} \delta LT$).
Lemmas~\ref{lem:noise models} or~\ref{lem:BV noise models} allow us to
bound the deviation of the continuous-time signal to its samples and
draw conclusions regarding the noise model~\eqref{eq:noise model} and
its counterpart~\eqref{eq:noise disc}.

\vspace{-0.3em}\begin{corollary}[Relations between noise models]
	\label{cor:noise models}
  Suppose $\delta$ is such that $\tfrac{T}{\delta}\in\Z_{>0}$ and let
  $L\geq 0$ be such that $w:[0,T]\rightarrow \R^n$ is either (i)
  $L$-square Lipschitz or (ii) $V_0^T(w)\leq \tfrac{1}{2}LT$.  Then,
  the following two statements hold:
\vspace{-0.3em}  \begin{align*}
    \delta W_\delta W_\delta^\top \leq Q
    & \Rightarrow \int_0^T
      w(t)w(t)^\top dt \leq Q + \tfrac{1}{2}\delta TLI_n,
    \\
    \int_0^T w(t)w(t)^\top dt \leq Q
    &
      \Rightarrow \delta W_\delta
      W_\delta^\top  \leq Q+\tfrac{1}{2}\delta TLI_n.	\vspace{-0.3em} 
  \end{align*}
\end{corollary}

\subsection{Inclusions between sets of consistent systems}
Here we address the second objective laid out in
Section~\ref{sec:problem} and compare the stabilizability properties
of the sets $\calZ(N_{\textup{cont}}(Q))$ and $\calZ(N_{\delta}(Q))$.
To tackle this, note that the additional assumptions on the noise
signal described in Section~\ref{sec:connection-noise-models} shrink
the set of systems consistent with the data and we formalize this
next. Given state and input trajectories $x:[0,T]\rightarrow\R^n$ and
$u:[0,T]\rightarrow\R^m$ that satisfy
Assumption~\ref{ass:state-input}, we define the sets
\vspace{-0.3em}\begin{align*}
  \calM_{x,u}^L
  & :=
    \lbrace(A,B)\in \R^{n\times n} \times
    \R^{n\times m} \mid \dot{x}-Ax-Bu
  \\
  &
    \quad \quad \textrm{ is } L\textrm{-square
    Lipschitz} \rbrace,
  \\ 
  \calN_{x,u}^L
  & :=
    \lbrace(A,B)\in \R^{n\times n} \times \R^{n\times
    m} \mid
  \\
  & \quad \quad V_0^T(\dot{x}-Ax-Bu) \leq
    \tfrac{1}{2}LT \rbrace. \vspace{-0.3em}
\end{align*}
Then, the set of all systems compatible with the measurement, the
noise model \eqref{eq:noise model}, and for which the noise is
$L$-square Lipschitz is
\begin{subequations}\label{eq:comp}
 \vspace{-0.3em} \begin{equation}\label{eq:comp+slc}
    \calZ(N_{\textup{cont}}(Q)) \cap
    \calM_{x,u}^L.\vspace{-0.3em}
  \end{equation}
  In a similar fashion, the set of all systems compatible with the
  measurement, the noise model \eqref{eq:noise model}, and for which
  the total square variation of the noise is less than or equal to
  $\tfrac{1}{2} LT$ is
\vspace{-0.3em}  \begin{equation}\label{eq:comp+bv}
    \calZ(N_{\textup{cont}}(Q)) \cap
    \calN_{x,u}^L.\vspace{-0.3em}
  \end{equation}
\end{subequations}
The true system from which the measurements are taken is contained in
the intersections in~\eqref{eq:comp} if the true realization of the
noise has the corresponding property. The following result is a
consequence of Corollary~\ref{cor:noise models}.

\vspace{-0.3em}\begin{corollary}[Inclusion relationships between sets of consistent
  systems]\label{cor:set inclusions}
  Let $x:[0,T]\rightarrow\R^n$ and $u:[0,T]\rightarrow\R^m$ be state
  and input trajectories satisfying
  Assumption~\ref{ass:state-input} and  let $\delta$ be such that
  $\tfrac{T}{\delta}\in\Z_{>0}$. Then
 \begin{itemize}
  \item{} [L-square Lipschitz noise:]
\vspace{-0.3em}    \begin{subequations}\label{eq:inclusions-Lipschitz}
      \begin{align}
        \calZ(N_\delta(Q)) \cap \calM_{x,u}^L
        &\subseteq \calZ(N_{\textup{cont}}(Q \!+ \! \tfrac{1}{2}\delta TLI_n)),
        \\
        \calZ(N_{\textup{cont}}(Q)) \cap \calM_{x,u}^L
        &\subseteq
          \calZ(N_\delta(Q+\tfrac{1}{2}\delta TLI_n)).\vspace{-0.4em} \label{eq:inclusions-Lipschitz 2}
      \end{align}
    \end{subequations}
    Moreover, if the noise signal corresponding to the measurements,
    $w:[0,T]\rightarrow \R^n$, is $L$-square Lipschitz, then the set
    on the left-hand side in~\eqref{eq:inclusions-Lipschitz 2} is
    non-empty and contains the true system.
  \item{} [Noise of bounded total square variation:]
\vspace{-0.3em}\begin{subequations}\label{eq:inclusions-variation}
      \begin{align}
        \calZ(N_\delta(Q)) \cap \calN_{x,u}^L
        &\subseteq
          \calZ(N_{\textup{cont}}(Q+ \tfrac{1}{2}\delta TLI_n)),
        \\
        \calZ(N_{\textup{cont}}(Q))\cap \calN_{x,u}^L
        & \subseteq \calZ(N_\delta(Q+\tfrac{1}{2}\delta TLI_n)).\vspace{-0.4em}  \label{eq:inclusions-variation 2}
      \end{align}
    \end{subequations}
    Moreover, if the noise signal corresponding to the measurements,
    $w:[0,T]\rightarrow \R^n$, is such that
    $V_0^T(w)\leq \tfrac{1}{2}LT$, then the set on the left-hand side
    in~\eqref{eq:inclusions-variation 2} is non-empty and contains the
    true system.
  \end{itemize}\vspace{-0.3em}
\end{corollary}

Recall that stabilizing controllers for the sets in the right-hand
sides of~\eqref{eq:inclusions-Lipschitz}
or~\eqref{eq:inclusions-variation} can be found using either
Theorems~\ref{thm:cont stab} or~\ref{thm:disc stab}. In particular,
this means that Corollary~\ref{cor:set inclusions} allows us to find a
stabilizing controller for all systems in~\eqref{eq:comp}.

\vspace{-0.3em}\begin{theorem}[Sufficient conditions for sampled data]
	\label{thm:lmi sampling implies cont}
  Consider state and input trajectories $x:[0,T]\rightarrow\R^n$ and
  $u:[0,T]\rightarrow\R^m$ such that Assumption~\ref{ass:state-input}
  holds. Suppose there exists $K\in\R^{m\times n}$,
  $P\in \R^{n\times n}$, and $\beta>\tfrac{1}{2}\delta TL$ such that
  $P>0$ and~\eqref{eq:LMI disc} is satisfied.
  Then, \eqref{eq:Lyap} holds for all
  $(A,B)\in \calZ(N_{\textup{cont}}(Q)) \cap \calM_{x,u}^L$ and all
  $(A,B)\in \calZ(N_{\textup{cont}}(Q)) \cap \calN_{x,u}^L$.
\end{theorem}

This result follows from combining Corollary~\ref{cor:set inclusions}
and Theorem~\ref{thm:disc stab}.
Theorem~\ref{thm:lmi sampling implies cont} provides conditions under
which the \emph{true} system can be stabilized and, importantly, this
can be checked with only samples of the measurements (i.e., the
conditions \emph{do not} require knowledge of the continuous-time
signals themselves).  In contrast, Theorem~\ref{thm:disc
  stab} similarly only relies on samples, but it only guarantees stabilization
of all systems in $\Sigma^\delta$. As discussed, for a given $\delta$, the
set $\calZ(N_\textup{cont})$ is \textit{not} necessarily contained in
$\Sigma^\delta$. Given that we cannot distinguish the true system from
any other system in $\calZ(N_\textup{cont})$, this means that
Theorem~\ref{thm:disc stab} might not guarantee the stabilization of
the true system. Comparing Theorems~\ref{thm:lmi sampling implies cont}
and~\ref{thm:disc stab}, we note that both require the satisfaction of
the same LMI, but that
Theorem~\ref{thm:lmi sampling implies cont} specifies
$\beta>\tfrac{1}{2}\delta TL$ instead of $\beta>0$. This can be
interpreted as requiring a \textit{margin of stability}, given
that~\eqref{eq:LMI disc} implies that
\vspace{-0.3em}\[
  (A+BK)P+P(A+BK)^\top < -\beta I_n < -\tfrac{1}{2}\delta TLI_n,\vspace{-0.3em}
\]
for all $(A,B)\in\calZ(N_\delta(Q))$. Theorem~\ref{thm:lmi sampling
  implies cont} can then be restated as follows: if the closed-loop
systems resulting from all systems compatible with the sampled
measurements are stable `enough', then all systems compatible with the
continuous measurements are stabilized as well.

\subsection{Verifying assumptions of Theorem~\ref{thm:lmi
			sampling implies cont}}%

A natural question arising from the result in Theorem~\ref{thm:lmi
  sampling implies cont} is: how small should the stepsize be to
ensure the samples from the continuous-time signals remain
informative?  Intuitively, if we sample very coarsely, e.g., with
$\delta =T$, then this will be unlikely.  The following result settles
this question.

\vspace{-0.3em}\begin{corollary}[Bound on stepsize for informativity of sampled
  data]\label{cor:how fine to sample}
  Consider state and input trajectories $x:[0,T]\rightarrow\R^n$ and
  $u:[0,T]\rightarrow\R^m$ such that Assumption~\ref{ass:state-input}
  holds. Assume the corresponding noise signal $w$ is either
  $L$-square Lipschitz or such that $V_0^T(w)\leq
  \tfrac{1}{2}LT$. Suppose that $(x,u)$ are informative for quadratic
  stabilization and let $\hat{\beta}$ be the largest $\beta>0$ such
  that there exists $K\in\R^{m\times n}$, $P\in \R^{n\times n}$, with
  $P>0$ and \eqref{eq:LMI stab}. If
  $\delta < \frac{1}{TL}\hat{\beta}$, then \eqref{eq:LMI disc} holds
  with
  $\beta=\hat{\beta}-\tfrac{1}{2}\delta TL > \tfrac{1}{2}\delta TL$.\vspace{-0.3em}
  %
\end{corollary}

The proof of this result leverages the margin of stability associated
to informative continuous-time data $(x,u)$ and follows from
Corollary~\ref{cor:set inclusions}.  As a consequence, we deduce that,
under the assumptions of Corollary~\ref{cor:how fine to sample}, there
\textit{always} exists a stepsize small enough to conclude quadratic
stabilization.

To draw conclusions regarding the true system using
Theorem~\ref{thm:lmi sampling implies cont}, we require that
$(A,B)\in\calM_{x,u}^L$ (resp. $(A,B)\in\calN_{x,u}^L$). The following
result identifies conditions to verify this on the basis of data.

\vspace{-0.3em}\begin{lemma}[Verifying the assumptions using data]
	\label{lem:SLC from data}
  Let the state $x:[0,T]\rightarrow\R^n$ and input
  $u:[0,T]\rightarrow\R^m$ trajectories satisfy
  Assumption~\ref{ass:state-input}. Suppose there exists
  $\lambda\geq 1$ such that
\vspace{-0.3em}  \begin{equation}\label{eq:bd AB}
    AA^\top +BB^\top < (\lambda-1)
    I_n,\vspace{-0.3em}
  \end{equation}
  for all $(A,B)\in\calZ(N_{\textup{cont}}(Q))$,
  \\
  (i) If 
  $\begin{pmatrix} \dot{x}^\top & -x^\top &
    -u^\top \end{pmatrix}^\top$ is $L$-square Lipschitz,~then
    \vspace{-0.3em}\begin{equation}\label{eq:ZinM}
      \calZ(N_{\textup{cont}}(Q))
      \subseteq \calM_{x,u}^{\lambda L}.\vspace{-0.3em}
    \end{equation}
    (ii) If
    $V_0^T \left(\begin{pmatrix} \dot{x}^\top & -x^\top &
        -u^\top \end{pmatrix}^\top\right)\leq \tfrac{1}{2}LT$, then
    \vspace{-0.3em}\begin{equation}\label{eq:ZinN}
      \calZ(N_{\textup{cont}}(Q)) \subseteq \calN_{x,u}^{\lambda L}.\vspace{-0.3em}
    \end{equation}
    Moreover, such $\lambda$ exists if and only if
  \vspace{-0.3em}\begin{equation}\label{eq:bounded int}
    \int_0^T \begin{pmatrix}
      x(t)\\u(t)\end{pmatrix}\begin{pmatrix}
      x(t)\\u(t)\end{pmatrix}^\top dt >0. \vspace{-0.3em}\end{equation}
\end{lemma} 
\begin{proof} 
  To prove statements (i) and (ii), let
  $ w_{(A,B)}(t) := \dot{x}(t)-Ax(t)-Bu(t)$.  Then, to prove that
  \eqref{eq:ZinM} hold, we need to show that $w_{(A,B)}(t)$ is
  $\lambda L$-square Lipschitz continuous for all
  $(A,B)\in\calZ(N_{\textup{cont}}(Q))$. Similarly, \eqref{eq:ZinN} is
  equivalent to $V_0^T(w_{(A,B)})\leq \tfrac{1}{2}\lambda LT$ for all
  $(A,B)\in\calZ(N_{\textup{cont}}(Q))$.  Note that $\eqref{eq:bd AB}$
  is equivalent to 
  \begin{equation}\label{eq:qmi bd AB}
   \begin{bmatrix} I_n & A & B\end{bmatrix}\begin{bmatrix} I_n & A & B\end{bmatrix}^\top
    < \lambda I_n, 
  \end{equation}
  for all $(A,B)\in\calZ(N_{\textup{cont}}(Q))$. This implies that $\norm{ \begin{bmatrix} I_n & A & B\end{bmatrix} } \leq \sqrt{\lambda}$. Note that
  \small \begin{align*} 
  	&w_{(A,B)}(t_1)w_{(A,B)}(t_1)^\top -w_{(A,B)}(t_2)w_{(A,B)}(t_2)^\top\\
    =\begin{bmatrix}I_n \\ A^\top \\ B^\top\end{bmatrix}^\top\!\!\!\!
    &\left(\!\!
      \begin{pmatrix}\dot{x}(t_1)\\ -x(t_1)\! \\-u(t_1)\! \end{pmatrix}\!\!
    \begin{pmatrix}\dot{x}(t_1)\\ -x(t_1)\! \\-u(t_1)\! \end{pmatrix}^{\!\!\top}
    \!\!-\! 
    \begin{pmatrix}\dot{x}(t_2)\\ -x(t_2)\! \\-u(t_2)\! \end{pmatrix}\!\!
    \begin{pmatrix}\dot{x}(t_2)\\ -x(t_2)\! \\-u(t_2)\! \end{pmatrix}^{\!\!\top}
    \right)\!\!\begin{bmatrix} I_n \\ A^\top \\ B^\top \end{bmatrix} 
  \end{align*}\normalsize
  Taking the norm on both sides of this equality, using the fact that
  matrix norms are sub-multiplicative, and applying this to the
  definition of $L$-square Lipschitzness (resp. total square
  variation) yields (i) (resp. (ii)). To prove the last statement, we
apply \cite[Cor. 4.13]{HJVW-MKC-JE-HLT:22} to see that \eqref{eq:qmi
  bd AB} holds for all $(A,B)\in\calZ(N_{\textup{cont}}(Q))$ iff
 there exists $\alpha\geq 0$ and $\beta>0$ such that
\vspace{-0.3em} \[
    \begin{bmatrix} \! (\lambda \!-\! 1 \!-\! \beta)I_n \! -\! \alpha
      Q \!\!\!& 0& 0 \\ 0&\!\! -I_n \!\!&0 \\
      0&0&\!\!-I_n\end{bmatrix}\! +
    \alpha \!\! \int_0^T\!\!\begin{pmatrix} \dot{x}(t) \\ \!-x(t)\! \\
      \!-u(t)\!\end{pmatrix}\!\!\!\begin{pmatrix} \dot{x}(t) \\
      \!-x(t)\! \\ \!-u(t)\! \end{pmatrix}^\top\!\!\! dt \!\geq\!  0.
  \]
  Zooming in on the right-lower block, we see that this
  requires~\eqref{eq:bounded int}. Conversely, if~\eqref{eq:bounded
    int} holds, there exists $\alpha$ such that
  \[
    \alpha \int_0^T \begin{pmatrix}
      x(t)\\u(t)\end{pmatrix}\begin{pmatrix}
      x(t)\\u(t)\end{pmatrix}^\top dt \geq I_{2n}.
  \]
  Then, for large enough $\lambda \ge 1$, the LMI is
  satisfied.
\end{proof} 


The combination of Lemma~\ref{lem:SLC from data} and
Theorem~\ref{thm:lmi sampling implies cont} yields the following
result.

\begin{corollary}[Sufficient conditions for informativity]
  Suppose that the state $x:[0,T]\rightarrow\R^n$ and input
  $u:[0,T]\rightarrow\R^m$ trajectories satisfy
  Assumption~\ref{ass:state-input} and that \eqref{eq:bounded int}
  holds. Take $\lambda\geq 1$ such that $\eqref{eq:bd AB}$ holds.
  Assume there exists $K\in\R^{m\times n}$, $P\in \R^{n\times n}$, and
  $\beta>\tfrac{1}{2}\delta TL$ such that $P>0$ and \eqref{eq:LMI
    disc} holds. If either (i) the signal
  $\begin{pmatrix} \dot{x}^\top & -x^\top &
    -u^\top \end{pmatrix}^\top$ is $\tfrac{L}{\lambda}$-square
  Lipschitz, or (ii)
  $V_0^T \left(\begin{pmatrix} \dot{x}^\top & -x^\top &
      -u^\top \end{pmatrix}^\top\right)\leq
  \tfrac{1}{2}\tfrac{L}{\lambda}T$, then $(x,u)$ are informative for
  quadratic stabilization.
\end{corollary}

\vspace{-1.2em}\subsection{Refining and coarsening sampled data}

Here we examine the impact of the stepsize on the informativity of
sampled data and its relationship with the informativity of the
continuous-time data.  As suggested by Corollary~\ref{cor:noise
  models}, decreasing the stepsize brings both notions of
informativity closer together. Instead, here we consider increasing
the stepsize and examine to what extent the number of samples can be
reduced while retaining informativity.

Let $\delta$ and $\gamma$ be stepsizes satisfying
$\tfrac{T}{\delta},\tfrac{T}{\gamma}\in\Z_{>0}$. Using the triangle
inequality and Lemma~\ref{lem:noise models}, we can conclude that, if
$w$ is $L$-square Lipschitz,
\[
  \norm{\delta W_\delta W_\delta^\top-\gamma W_\gamma W_\gamma^\top}
  \leq \tfrac{1}{2}(\delta+\gamma) TL.
\]
This result can be applied similarly to Corollary~\ref{cor:noise
  models} to obtain results comparing the respective noise models, in
turn linking their respective informativity properties.  However, if
we refine (resp., coarsen) the sampling by multiplying the stepsize
with a constant, we can obtain less conservative bounds.

\begin{lemma}[Bounds on noise model under different stepsizes]
  Let $w:[0,T]\rightarrow \R^n$ be $L$-square Lipschitz, $\delta$ and
  $\gamma$ such that $\gamma =(\ell+1)\delta$ with
  $\tfrac{T}{\delta},\tfrac{T}{\gamma}\in\Z_{>0}$ and
  $\ell\in\Z_{\geq0}$.~Then
  \[
    \norm{\delta W_\delta W_\delta^\top-\gamma W_\gamma W_\gamma^\top}
    \leq \tfrac{1}{2}(\gamma-\delta) TL = \tfrac{1}{2}\ell\delta TL.
  \] 
\end{lemma}
\begin{proof} 
  Note that
 \vspace{-0.3em} \[
    W_\delta W_\delta^\top = \sum_{k=0}^{\sfrac{T}{\gamma}-1}
    \sum_{j=0}^{\ell} w(k\gamma+j\delta)w(k\gamma+j\delta)^\top.\vspace{-0.3em}
  \]
  On the other hand, we can expand
 \vspace{-0.3em} \[
    (\ell+1) W_\gamma W_\gamma^\top = \sum_{k=0}^{\sfrac{T}{\gamma}-1}
    \sum_{j=0}^{\ell} w(k\gamma)w(k\gamma)^\top.\vspace{-0.3em}
  \]
  Since  $w$ is $L$-square Lipschitz, 
  $ \norm{ w(k\gamma+j\delta)w(k\gamma+j\delta)^\top -
    w(k\gamma)w(k\gamma)^\top } \leq j\delta L$.  Combining the above,
  we get
  \[
    \norm{\delta W_\delta W_\delta^\top- \gamma W_\gamma
      W_\gamma^\top} \leq \delta\tfrac{T}{\gamma}
    \Big(\sum_{j=0}^{\ell} j\delta L \Big)=
    \tfrac{1}{2}(\gamma-\delta) TL,
  \]
  proving the result.
\end{proof}

This result allows us to link properties of the noise models under
different sampling rates.

\begin{corollary}[Relations between noise models with different
  stepsizes]
  Let $w:[0,T]\rightarrow \R^n$ be $L$-square Lipschitz, $\delta$ and
  $\gamma$ such that $\gamma =(\ell+1)\delta$ with
  $\tfrac{T}{\delta},\tfrac{T}{\gamma}\in\Z_{>0}$, and
  $\ell\in\Z_{\geq0}$.~Then
  \begin{align*}
    \delta W_\delta W_\delta^\top \leq Q
    &\Rightarrow \gamma W_\gamma
      W_\gamma^\top \leq Q +
      \tfrac{1}{2}(\gamma-\delta)
      TLI_n,
    \\		
    \gamma W_\gamma W_\gamma^\top\leq Q
    &\Rightarrow \delta W_\delta
      W_\delta^\top  \leq Q +
      \tfrac{1}{2}(\gamma-\delta)
      TLI_n.	 
  \end{align*} 
\end{corollary}

We are now ready to provide a criterion to increase the sampling
stepsize without losing informativity.

\begin{theorem}[Coarsening measurements]\label{th:coarsening}
  Consider state $x:[0,T]\rightarrow\R^n$ and input
  $u:[0,T]\rightarrow\R^m$ trajectories such that
  Assumption~\ref{ass:state-input} holds. Assume the corresponding
  noise signal $w$ is $L$-square Lipschitz. Suppose that the data
  $(\dot{X}_\delta,X_\delta,U_\delta)$ are informative for
  continuous-time quadratic stabilization and let $\hat{\beta}$ the
  largest $\beta>0$ such that there exists $K\in\R^{m\times n}$,
  $P\in \R^{n\times n}$, with $P>0$ and \eqref{eq:LMI disc}. Then, the
  data $(\dot{X}_\gamma,X_\gamma,U_\gamma)$ are informative for
  continuous-time quadratic stabilization for
  $\gamma =(\ell+1)\delta$, with
  $\ell< \tfrac{2}{\delta TL}\hat{\beta}$.
\end{theorem}

Note that, under the assumptions of Theorem~\ref{th:coarsening}, the
samples $(\dot{X}_\gamma,X_\gamma,U_\gamma)$ are contained in those of
$(\dot{X}_\delta,X_\delta,U_\delta)$. This means that, given
informative data, the result allows to find a subset of it which
remains informative. In particular, to determine continuous-time
quadratic stabilization, we can draw conclusions from data that
contains $\ell$ times less samples.  One can derive similar results
for the case of noise with bounded total variation, but we omit them
for brevity.

\vspace{-0.8em}\section{Scalar system with square Lipschitz
  noise}\label{sec:example} To visualize the results and show
  that the effects described are important to take into account, we
  provide a simple example. We show here that the nontrivial effects
  of sampling arise even for a scalar system with well-behaved noise
  and input signals. Consider the scalar linear system
\[
  \dot{{x}}(t) = - {x}(t) +\tfrac{1}{10} {u}(t) +{w}(t),
\] 
with initial condition $x(0)=1$. The time horizon is $T=1$.  We
consider noise signals of the form \eqref{eq:noise model} with
$Q=1$. We excite the system with a uniform input $u(t) = 1$ and the
(piecewise linear) noise signal
\[
  w(t) = \max\{ 0, 2-4t\} = \begin{cases} 0 & t\leq \tfrac{1}{2} \\
    2-4t & t> \tfrac{1}{2} \end{cases}.
\]

It is straightforward to show that
$ \int_{0}^1 w(t)^2 dt = \tfrac{2}{3}\leq 1$, and that $w$ is
16-square Lipschitz. Solving for the dynamics yields
\[
  x(t) = \begin{cases} \tfrac{1}{10}e^{-t}(9+e^t)  & t\leq
    \tfrac{1}{2}\\ \tfrac{1}{10}e^{-t}(9-40\sqrt{e}+e^t(61-40t) &
    t>\tfrac{1}{2} \end{cases}.
\]
Figure~\ref{fig:state+derivative} shows the signal $x$ and its
derivative.

\begin{figure}[t]
  \begin{tikzpicture}[scale=0.55]
    \begin{axis}[axis lines = left]
      \addplot [domain=0:0.5, samples=100, color=blue	]
      {(1/10)*exp(-x)*(9+exp(x))};
      \addplot [domain=0.5:1, samples=100, color=blue	]
      {(1/10)*exp(-x)*(9-40*sqrt(e)+exp(x)*(61-40*x))};
    \end{axis}
  \end{tikzpicture}\!\!\!\!\!\!
  \begin{tikzpicture}[scale=0.55]
    \begin{axis}[axis lines = left, ymin=-2, ymax=0	]
      \addplot [domain=0:0.5, samples=100, color=red]
      {-(1/10)*exp(-x)*9};
      \addplot [domain=0.5:1, samples=100, color=red]
      {-4+4*exp(1/2-x)-(1/10)*exp(-x)*9};
    \end{axis}
  \end{tikzpicture}\hspace{-1.1cm}
  \caption{Measured state $x(t)$ (left) and derivative $\dot{x}(t)$
    (right) signals. These, along with $u(t)=1$, are the
    continuous-time data considered in Section~\ref{sec:example}.}
  \label{fig:state+derivative}
  \vspace{-1.5em}
\end{figure}
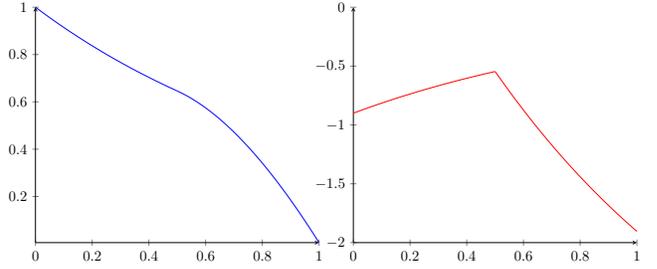

\begin{figure*}[t!]
  \includegraphics[width=.24\linewidth]{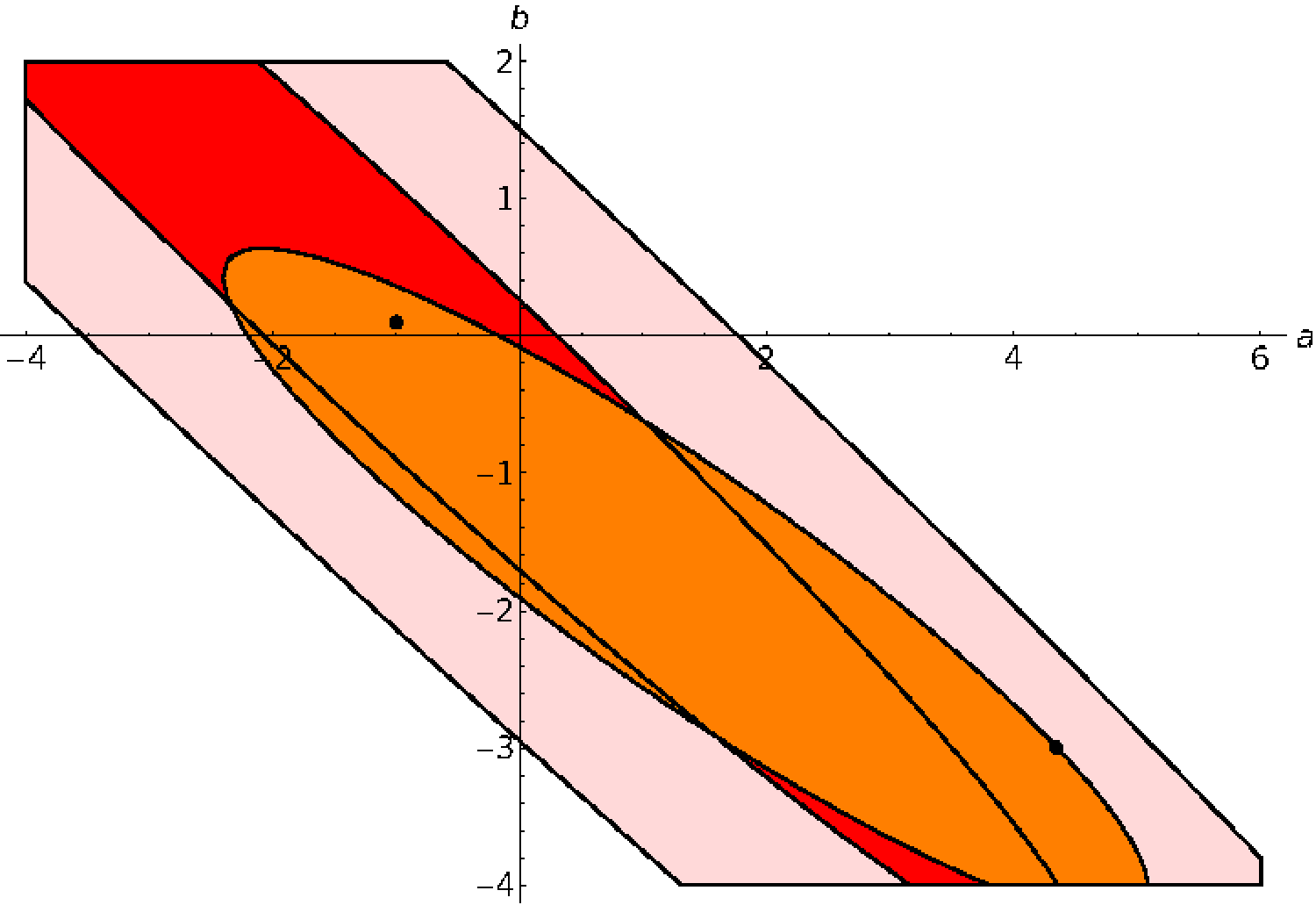}
  \includegraphics[width=.24\linewidth]{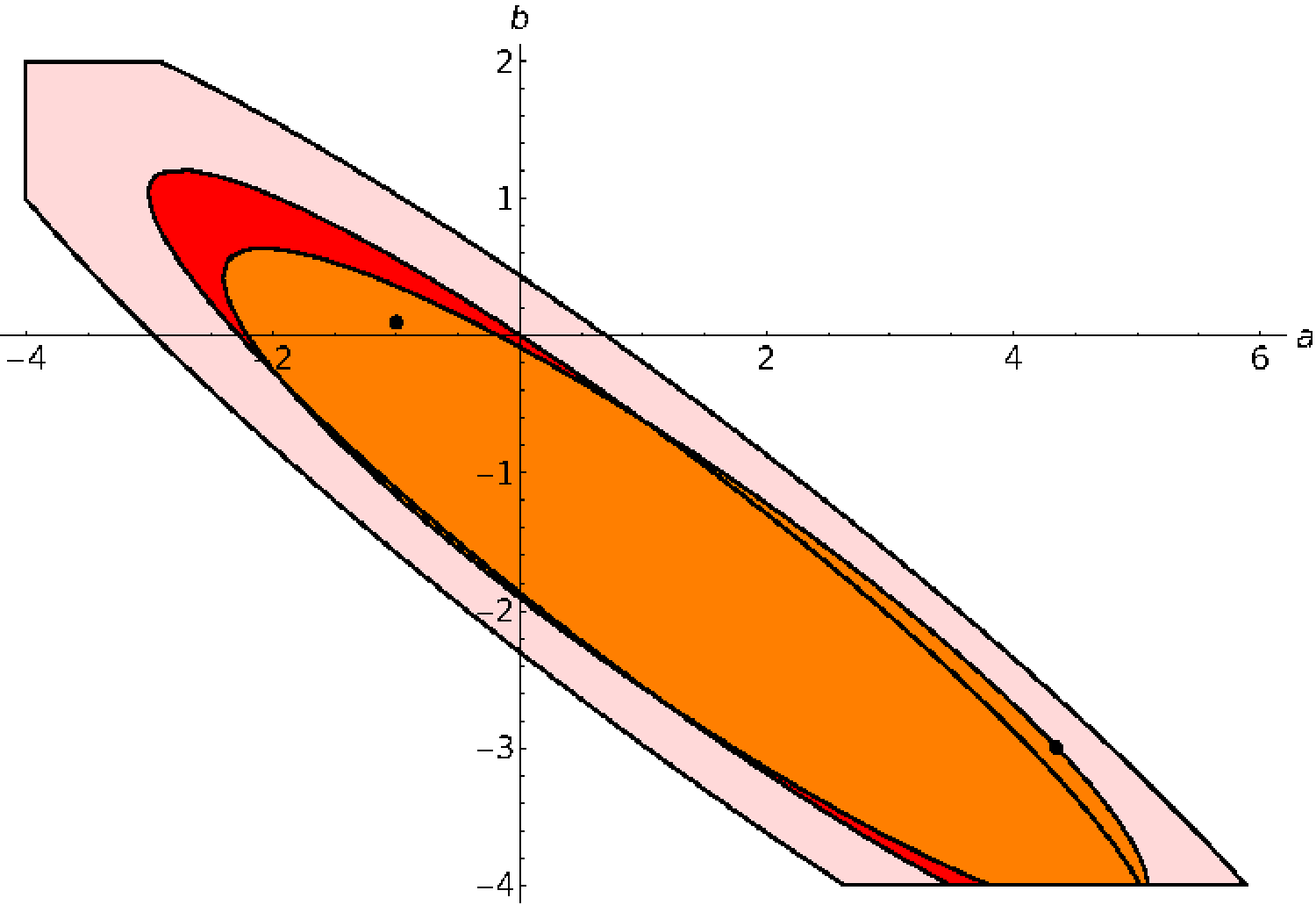}
  \includegraphics[width=.24\linewidth]{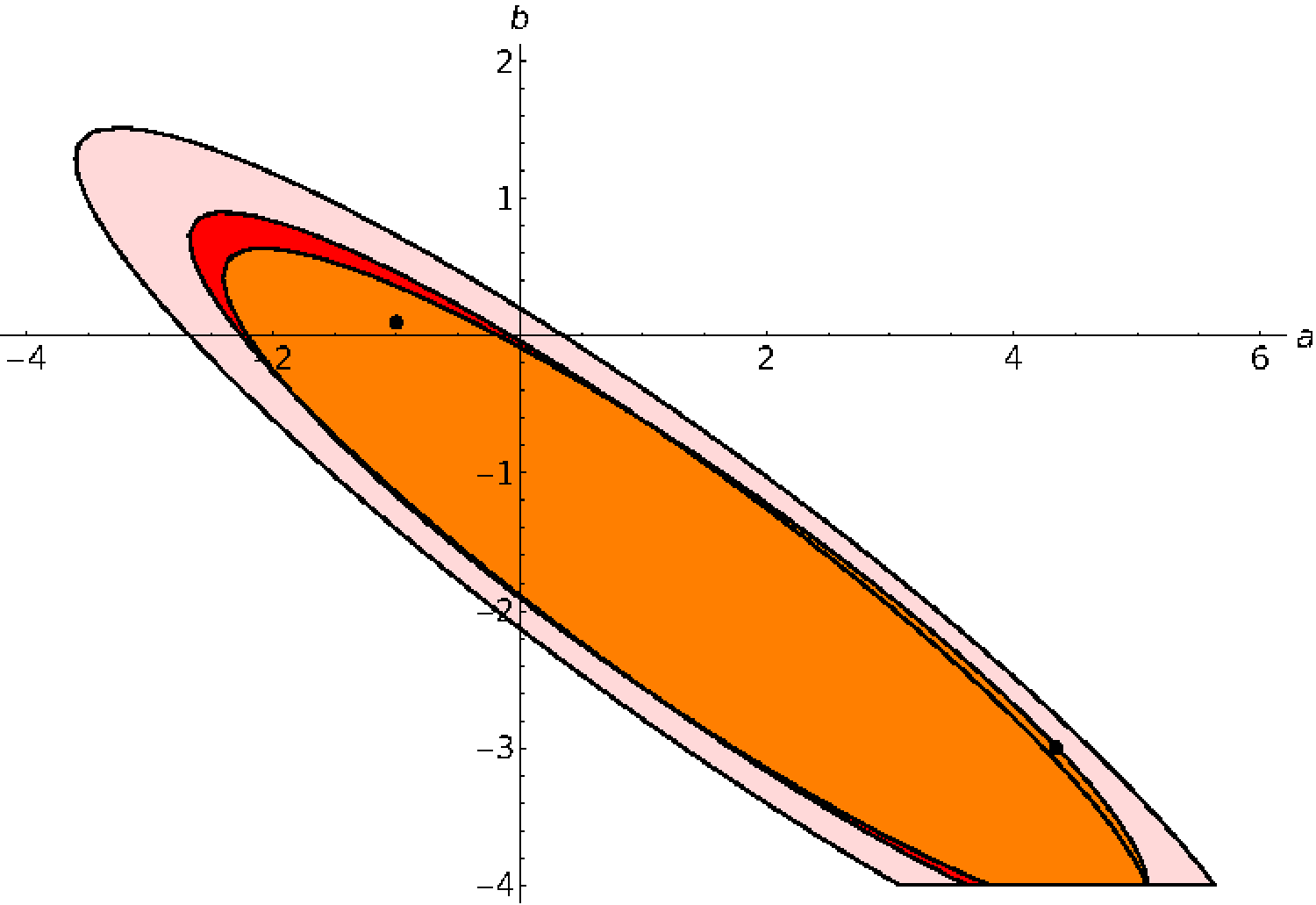}
  \includegraphics[width=.24\linewidth]{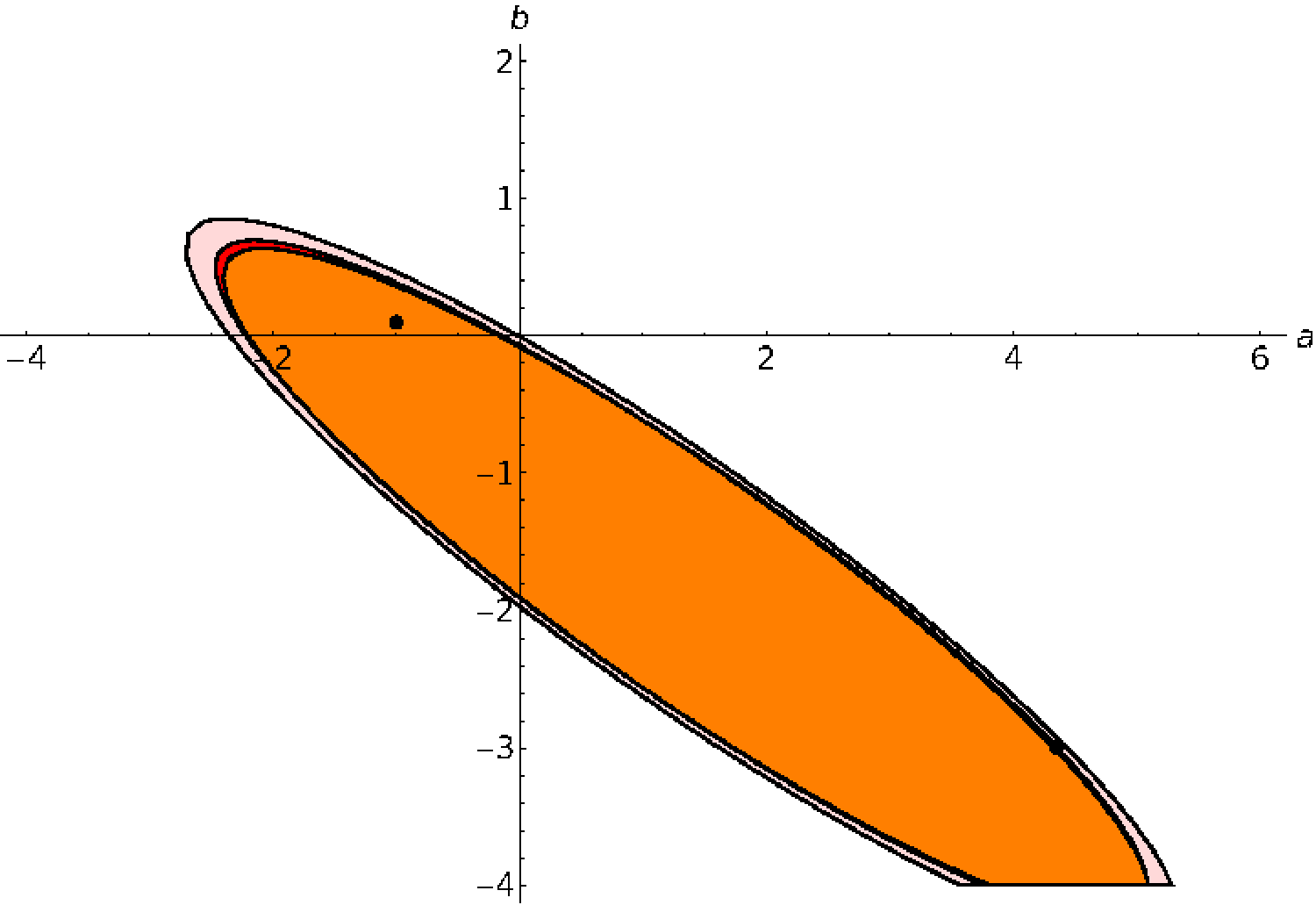}
  \caption{The sets of systems $(a,b)$ that are compatible with the
    measurements and with sampled data. In orange the set
    $\calZ(N_{\textup{cont}}(1))$. In red the set
    $\calZ\left(N_\delta(1)\right)$ for $\delta=\tfrac{1}{2}$,
    $\tfrac{1}{8}$, $\tfrac{1}{16}$, and $\tfrac{1}{64}$ from left to
    right. In light red the set
    $\calZ\left(N_\delta(1+\tfrac{1}{2}\delta TL)\right)$ for the same
    values of $\delta$, $T=1$, and $L=16$. The black dots denotes the
    true system $(-1,\tfrac{1}{10})$ and the (indistinguishable on the basis of the measurements) system $(4.35,-3)$.}\label{fig:level sets}
\end{figure*}	

Any system, given in terms of state and input matrices $(a,b)$, is
compatible with the measurements if and only if
$(a,b) \in \calZ(N_{\textup{cont}}(1))\cap \calM_{x,u}^{16}$, where
$N_{\textup{cont}}(1)$ is given in \eqref{eq:N noise}. Calculating the
relevant integrals yields
\[
  N_{\textup{cont}}(1) \approx \begin{bmatrix} -0.154 & -0.500 &
    -0.995 \\ -0.500 & - 0.422 & - 0.595\\ -0.995 &- 0.595&-
    1\end{bmatrix}.
\]
Now, note that for $P=\tfrac{1}{2}>0$, $K=2$, and
$\beta =\frac{1}{10}$, the LMI~\eqref{eq:LMI stab} holds. Using
Theorem~\ref{thm:cont stab}, this allows us to conclude that the data
$(x,u)$ is informative for quadratic stabilization. Indeed, the true,
measured system is stabilized by $K=2$.

Next, we turn our attention to sampling the data. We take $\delta$
equal to $2^{-i}$, for $i=1,\ldots 6$, and show the corresponding
matrices $N_\delta(1)$ in~\eqref{eq:N matrices example}.
\begin{figure*}
 	\vspace{-1.5em}
  \footnotesize
  \begin{subequations}\label{eq:N matrices example}
    \begin{align} 
      &N_{\tfrac{1}{2}}(1) \!\! &\approx \begin{bmatrix}0.446&-0.626&-0.723\\-0.626&-0.709&-0.823\\-0.723&-0.823&-1 \end{bmatrix}, \hspace{1em}
                                                            &N_{\tfrac{1}{4}}(1) \!\!&\approx \begin{bmatrix}0.171&-0.588&-0.864\\-0.588&-0.557&-0.714\\-0.864&-0.714&-1 \end{bmatrix}, \hspace{1em}
                                                            &N_{\tfrac{1}{8}}(1) \!\!&\approx \begin{bmatrix}0.0152&-0.550&-0.931\\-0.550&-0.487&-0.656\\-0.931&-0.656&-1 \end{bmatrix}, \\
      &N_{\tfrac{1}{16}}(1) \!\!&\approx \begin{bmatrix}-0.068&-0.526&-0.963\\-0.526&-0.454&-0.626\\-0.963&-0.626&-1\end{bmatrix},\hspace{1em}
                                                            &N_{\tfrac{1}{32}}(1) \!\!&\approx \begin{bmatrix}-0.111&-0.514&-0.979\\-0.514&-0.438&-0.610\\-0.979&-0.610&-1\end{bmatrix},\hspace{1em}
                                                            &N_{\tfrac{1}{64}}(1) \!\!&\approx \begin{bmatrix}-0.132&-0.507&-0.987\\-0.507&-0.430&-0.603\\-0.987&-0.603&-1 \end{bmatrix}.
		\end{align}
	\end{subequations}
	\hrulefill 
	\vspace{-2.5em}
\end{figure*}
We first consider whether the samples are informative for
continuous-time quadratic stabilization. Note that, for each
$i\leq 3$, the left-upper block of $N_{2^{-i}}(1)$ is greater than
0. This implies that $(0,0)\in\calZ(N_{2^{-i}}(1))$, and therefore the
data cannot be informative for continuous-time quadratic
stabilization. Figure~\ref{fig:level sets} illustrates this, showing
the sets of systems consistent with the continuous measurements and
with sampled data for $\delta=\tfrac{1}{2}$, $\tfrac{1}{8}$,
$\tfrac{1}{16}$, and $\tfrac{1}{64}$.  Using Matlab with YALMIP
\cite{JL:04} and MOSEK, we can check the conditions in
Theorem~\ref{thm:disc stab} for different values of $\delta$. This
yields that the data are informative for continuous-time quadratic
stabilization for $\delta=\tfrac{1}{16}$ and smaller values.  As
argued above, this does not yet allow us to conclude that the
continuous-time measurements are informative for quadratic
stabilization of the true system on the basis of sampled data. To
illustrate this, recall that, on the basis of the measurements
$(x,u)$, we cannot distinguish the true system from any of those in
$\calZ(N_{\textup{cont}}(1))$. In Figure~\ref{fig:level sets}, we see
that the system $(4.35,-3)$, for example, is compatible with the
continuous measurements, but $(4.35,-3) \not \in \calZ \left(N_\delta(1)\right)$. This shows
that even if all systems in $\calZ\left(N_\delta(1)\right)$ can be
stabilized, this does not imply that the measurements $(x,u)$ are
informative for quadratic stabilizability.

To determine for the stepsizes for which sampled versions of the
continuous-time measurements are informative for quadratic
stabilization of the true system, we employ the additional knowledge
on the noise signal and resort to Theorem~\ref{thm:lmi sampling
  implies cont}.  In this case, the fact that $w$ is $L$-square
Lipschitz with $L=16$ (alternatively, a more conservative bound for
$L$ could be obtained from Lemma~\ref{lem:SLC from data}).
Note in particular that the set inclusions displayed in
Figure~\ref{fig:level sets}, where $\calZ(N_{\textup{cont}}(1))$ is
contained in each of the sets
$\calZ\left(N_\delta(1+\tfrac{1}{2}\delta TL)\right)$, are consistent
with~\eqref{eq:inclusions-Lipschitz} in Corollary~\ref{cor:set
  inclusions}.

Using Matlab, we verify that the required LMI of Theorem~\ref{thm:lmi
  sampling implies cont} is feasible for $\delta=\tfrac{1}{64}$. This
guarantees the existence of a stabilizing feedback $K$ for all systems
$(a,b)\in \calZ(N_{\textup{cont}}(1)) \cap \calM_{x,u}^L$ on the basis
of sampled data with $\delta=\tfrac{1}{64}$. This is consistent with
the bound for the stepsize obtained in Corollary~\ref{cor:how fine to
  sample}, which guarantees samples from the continuous-time signals
are informative for
$\delta<\tfrac{1}{16}\hat{\beta}\approx 0.0096 \approx
\tfrac{1}{104}$.

\vspace{-0.5em}
\section{Conclusions}\label{sec:conc}

We have studied the informativity problem for continuous-time signals
and systems. We first characterized when continuous-time data is
informative for continuous-time stabilization and then focused on
understanding the informativity of sampled data.  After motivating the
need for additional assumptions on the noise signal, we have
introduced the notions of square Lipschitzness and bounded total
square variation.  Under these noise models, we have provided
sufficient conditions for stabilizability properties of the set of
systems compatible with the continuous-time measurements on the basis
of sampled data and characterized the role of the sampling
stepsize. These results provide a stepping stone towards a full
  treatment of continuous systems on the basis of samples. Future
  research will include the investigation of necessary conditions, the
  effect of the estimation of derivative from state samples, the study
of informativity under other noise models, and the generalization of
our results to problems beyond stabilization like $\calH_2$ and
$\calH_\infty$ performance.

\bibliography{../bib/alias,../bib/JC,../bib/Main,../bib/Main-add,../bib/New,../bib/FB}


\bibliographystyle{IEEEtran}

\end{document}